\documentclass[12pt]{amsart}

\usepackage{dsfont}
\usepackage{hyperref}
\usepackage{orcidlink}
\usepackage{url}
\usepackage{epsfig}
\usepackage{amsmath}
\usepackage{mathdots}
\usepackage[matrix,arrow,curve]{xy}
\usepackage{mathtools}
\usepackage{amsfonts}
\usepackage{amssymb}
\usepackage{amscd}
\usepackage[matrix,arrow]{xy}
\usepackage{amsmath,amsthm,amsfonts,bbm, amsthm}

\usepackage{mathrsfs}

\usepackage{color}
\usepackage{bbm}
\usepackage{tikz}
\usepackage{tikz-cd}
\usepackage{textcomp}
\usepackage{geometry}
\usepackage{verbatim}

\theoremstyle{plain}
\newtheorem{theor}{Theorem}[section]
\newtheorem{lemma}[theor]{Lemma}
\newtheorem{prop}[theor]{Proposition}

\newtheorem{cor}[theor]{Corollary}

\numberwithin{equation}{section}

\theoremstyle{remark}

\newtheorem{rem}[theor]{Remark}

\theoremstyle{definition}

\newtheorem{defin}[theor]{Definition}

\def \EE {\mathcal{E}}

\def \cO {\mathcal{O}}

\def \LL {\mathcal{L}}

\def \R {\mathbb{R}}
\def \Q {\mathbb{Q}}

\def \Z {\mathbb{Z}}
\def \ZZ {\mathbb{Z}}

\def \NN {\mathfrak{N}}

\def \A {\mathbb{A}}

\def \GG {\mathcal{G}}

\newcommand{\Image}{\mathrm{Im}}

\newcommand {\TSym}{\mathrm{TSym}}

\newcommand{\Loc}{\mathrm{Loc}}

\newcommand{\bH}{\mathbf{H}}
\newcommand{\Iw}{\mathrm{Iw}}
\newcommand{\AL}{\mathrm{AL}}
\newcommand{\mom}{\mathrm{mom}}
\newcommand{\Res}{\mathrm{Res}}
\newcommand{\et}{\mathrm{\acute et}}

\newcommand{\arrow}[2]{\langle #1\to #2 \rangle}

\def \ge {\geqslant}

\def \L  {\mathcal L}

\def \Id {{\rm id}}

\def \Res {{\rm Res}}

\def \Gal {{\rm Gal}}

\def \mom {{\rm mom}}

\def \GL {{\rm GL}}

\def \Frac {{\rm Frac}}

\def \AL {{\rm AL}}

\def \pr {{\rm pr}}

\title{Poincar\'e duality in Hida families}
\author{Dmitrii Krekov\,\orcidlink{0000-0002-7771-3864}}
\address{Department of Mathematics, ETH Zurich, R\"amistrasse 101, 8092 Z\"urich, Switzerland}
\email{dmitrii.krekov@math.ethz.ch}

\begin{document}
\maketitle
\begin{abstract}
    We construct a bilinear pairing for $p$-adic families of cohomology classes associated to locally symmetric spaces of a certain level at $p$ and obtain a comparison between the specialisations of constructed pairing and the natural duality pairings on cohomology of automorphic sheaves under specialisation maps. We show that the considered symmetric spaces with the chosen level admit an automorphism which naturally generalises the Atkin--Lehner involution for modular curves and interchanges the ordinary and anti-ordinary parts of the cohomology thereby making it possible to pair two anti-ordinary classes. Our construction is motivated by the pairing constructed by Ohta for modular curves and recovers it.
\end{abstract}

\tableofcontents

\section{Introduction}

\renewcommand{\thetheor}{\Alph{theor}}

The cohomology groups of modular curves and more generally, of locally symmetric spaces, carry two fundamental structures: Poincar\'e duality on the one hand, and the action of the Hecke algebra on the other. 

Exploiting Hecke operators at a prime  $p$, as was firstly shown in seminal works of Hida starting from \cite{Hida1986}, one organises (ordinary) cohomology classes of varying weight into $p$-adic families over weight space. It is therefore natural to ask, whether Poincar\'e duality itself varies $p$-adically: given two $p$-adic
families of cohomology classes, can one pair them with each other so that the result interpolates the duality pairings of their specialisations at
classical weights?

For modular curves this question was answered by Ohta \cite{OhtES}. The
purpose of the present note is to generalise his construction to locally symmetric spaces associated to more general connected reductive groups in the framework of $p$-adic families of
cohomology classes developed by
Loeffler--Rockwood--Zerbes \cite{LRZ}. We begin by recalling Ohta's theorem, which serves as a motivation for our main result.

\subsection{Motivation: Ohta's $\Lambda$-adic duality for modular
curves}\label{ssec:ohta-intro}

Fix a prime $p$ and an integer $N\geq 1$ coprime to $p$, and for $r\geq 1$
let $Y_r$ denote the modular curve of level $\Gamma_1(Np^r)$
over $\Q$. The (compactly supported) \'etale cohomology groups of the tower $\{Y_r\}_r$ are
connected by trace maps, and the inverse limits
\[
  \bH_{(c)} \;:=\; \varprojlim_r H^1_{\et,(c)}\bigl(Y_{r,\overline\Q},\Z_p\bigr)
\]
are modules over the Iwasawa algebra $\Lambda = \Z_p[[1+p\Z_p]]$, the
group $\Z_p^\times$ acting through the diamond operators at $p$. Hida's anti-ordinary projector $e' = \lim_n (U_p')^{n!}$ cuts out the maximal direct summand
of $\bH_{(c)}$ on which $U_p'$ acts invertibly and by the control theorems of Hida and Ohta, $e'\bH$ is free of finite rank over $\Lambda$. 

Moreover, denoting by $\EE$ the universal elliptic curve (the base being any of the curves in the tower and clear from the context in what follows) there is a canonical section of $\EE[p^r]$ over $Y_r$, compatible with $[p]$ on $\EE$ and transition maps in the tower, so for any natural $k$ one gets a map $$H^1_{\et,(c)}\bigl(Y_{r,\overline\Q},\Z_p/(p^r)\bigr)\to H^1_{\et,(c)}\bigl(Y_{1,\overline\Q},\TSym^k (\EE[p^r])).$$
Hence, passing to the inverse limit provides for every integer $k\geq 0$ the specialisation (``moment'') map
$$\mom^{k}\colon \bH_{(c)}\to H^1_{\et,(c)}\bigl(Y_{1,\overline\Q},\TSym^k(T_p\EE)\bigr)$$ to the cohomology of the single curve $Y_1$.

Choose a basis of $\Z_p(1)$ and denote for brevity $Y=Y_{1,\overline{\Q}}$. Then the Weil pairing on $T_p\EE$ and Poincar\'e duality on $Y$ provide for every $k$ a pairing
\[
  (\cdot,\cdot)\colon H^1_c\bigl(Y,\TSym^k\bigr)\times H^1\bigl(Y,\TSym^k\bigr)
  \longrightarrow H^2_c(Y,\Z_p),
\]
and one may ask whether these pairings interpolate to a $\Lambda$-bilinear
pairing on $e'\bH_c\times e'\bH$. Incompatibility of trace maps with the Poincar\'e duality is what makes the question non-trivial. Alternatively, the problem arises from the following observation: the adjoint of $e'$
under the cup product is the ordinary projector, it is associated to the transpose operator $U_p$, which is intertwined with $U'_p$ via the Atkin--Lehner involution $W_{Np}$. So on the finite level the duality identifies $e'H^1_{c}(\cdots)$ with the dual of the \emph{ordinary} part
$eH^1(\cdots)$ --- the part on which $U_p$ acts invertibly --- and not with the dual
of $e'H^1$ itself.

This also hints at a possible solution: since $W_{Np}$ intertwines $U_p$ and $U_p'$ (up to Hecke operators away from $p$), and hence interchanges ordinary and anti-ordinary parts, one can apply the classical trick and twist the pairing by $W_{Np}$: the assignment $(x,y)\mapsto (x, W_{Np}\, y)$ restricts to a non-degenerate pairing on ordinary parts and is bilinear with respect to Hecke action. Ohta proved that these twisted pairings indeed interpolate:
there is a perfect $\Lambda$-bilinear pairing between $e'\bH_c$ and $e'\bH$,
with values in a free $\Lambda$-module of rank one, whose specialisation at
each integer weight of $\Lambda$ recovers the cup product twisted by the
Atkin--Lehner operator and normalised by an explicit scalar (see \cite[Theorem~4.2.5]{OhtES} for the version with cuspidal cohomology and \cite[Theorem~1.3.3]{Ohta2003} for the presented version).

\subsection{Outline of the results}\label{ssec:main-results-intro}

We generalise the picture above beyond modular curves to other locally symmetric spaces which requires some inputs: note that for a general reductive group $G$ and a general level subgroup at $p$ there is no direct analogue of the Atkin--Lehner operator $W_p$ so one must exhibit levels which admit one. Once these are constructed we consider a locally symmetric space $Y$ and define Atkin--Lehner operators $\AL$ for cohomology of automorphic local systems on it which preserve integral structures, this amounts to certain scalar normalisations (note that in existing literature there is also an ambiguity in normalisations for the modular curve case). Using trace maps along certain infinite towers one similarly defines groups $H^{\bullet}_{Iw,(c)}$ which are analogues of Ohta's $\bH$. They specialise to the cohomology of $Y$ with coefficients in automorphic local systems (with coefficients in a finite $\Z_p$-algebra $\cO$) via moment maps. These local systems are indexed by certain weights $\lambda$ of a maximal torus $T$ of $G$ which are trivial on its certain closed subgroup $E$ (see Section \ref{ssec:notations} for precise notations and setup we work in)  and our main result is the following:
 \\

\begin{theor}
    There is a commutative diagram

\begin{tikzcd}[column sep=huge]
H^{\bullet}_{Iw,c}\otimes  H^{d-\bullet}_{Iw} \arrow [r,"{{\mom}^{\lambda},{\mom}^{\lambda^{\vee}}}"]\arrow[d,"d_{Iw}"]& H_c^{\bullet}(Y,\LL_{\lambda})\otimes H^{d-\bullet}(Y,\LL_{\lambda}^{\vee}) \arrow[d,"{(\cdot,\AL\cdot)}"]\\
\cO[[T_0/E]]\otimes H_{c}^d(Y,\cO) \arrow[r,"c_{\lambda}\cdot\mom^{\lambda}"]& H_c^d(Y,\cO),
\end{tikzcd}

with explicit constants $c_{\lambda}$. 
\end{theor}

Here $T_0=T(\Z_p)$ and $\mom^{\lambda}$ by the lower horizontal arrow corresponds to the homomorphism $\cO[[T_0/E]]\to \cO$ induced by the character $\lambda$. Furthermore, after twisting the natural
$\cO[[T_0/E]]$-module structure on one of the two factors by the involution
$t\mapsto w_0(t^{-1})$ where $w_0$ is the longest Weyl element, the pairing $d_{\Iw}$ becomes $\cO[[T_0/E]]$-bilinear. 

The result actually holds for arbitrary $d$, although considering top geometric or arithmetic degree seems to be most relevant for applications. We note that the statement is valid for various cohomology theories, e.g. Betti, geometric or absolute \'etale (when the latter two make sense). Furthermore, applying anti-ordinary projection one obtains similar results with anti-ordinary parts of the cohomology groups involved (see Corollary~\ref{dualityord}).

\subsection{Examples and applications}\label{ssec:applications-intro}

Ohta's duality underlies his study of the $p$-adic Eichler--Shimura isomorphism for $\Lambda$-adic cusp forms and his subsequent work on the Eisenstein ideal (see references above). It also enters Sharifi's work \cite{MR2753604} and in particular the formulation of Sharifi's conjectures as well as their study by Fukaya and Kato (see \cite{MR4718483}). We hope that our results will find a wide range of applications as well and mention some current ideas which the author intends to explore in subsequent works. Section \ref{exappl} of this note illustrates our results in two examples.

For $G = \mathrm{GL}_{2,\Q}$, we relate our results to Ohta's pairing; the constants $c_{k} = p^{k}$ match Ohta's normalisation of the Atkin--Lehner operator. (Ohta works with the operator $W_{Np}$ rather than $W_p$ but this discrepancy is harmless, see Remark \ref{Ohtarem}).

The second example is the source of the applications we have in mind. Let
$L/F$ be a quadratic extension of totally real fields with $L$ unramified at
$p$, let $[F:\Q]=d$ and $G = \Res_{L/\Q}\mathrm{GL}_2$. Cauchi, Nicole and
Rosso \cite{CaNiRo} construct a class
$\mathscr{Z}_\infty$ in the middle degree Iwasawa cohomology of Hilbert modular variety interpolating Hirzebruch--Zagier cycles. Applying our main Theorem  and a projection to the $\mathbbm{f}$-isotypical part, for $\mathbbm{f}$ a Hida family of Hilbert cusp forms, to produce a compactly supported class, we show
(Proposition~3.5) that the intersection numbers of $\AL(\mathscr{Z}_\infty)$ with the
$\mathbbm{f}$-part of $\mathscr{Z}_\infty$, normalised by the constants
$c_\lambda$ which are given explicitly, interpolate $p$-adically over the even parallel self-dual weights. This
example is motivated by the discussion following \cite[Definition~7.10]{CaNiRo}. If one had a compactly supported version $\mathscr{Z}^c_\infty$ of $\mathscr{Z}_{\infty}$ our result provides a way to pair it against its Atkin--Lehner image and hence study congruences as is menctioned in op.\ cit.\

In the upcoming work we also plan to apply our results to construct an explicit basis $\underline{\eta}$ of the period module of
\cite[Definition~11.3.2]{grossi2025asaiflachclassespadiclfunctions} over the diagonal of the weight space around the base change point using de Rham classes associated to $pr_{\mathbbm{f}}(\mathscr{Z}_{\infty})$, and
thereby to normalise the three-variable $p$-adic Asai $L$-function of
op.\ cit.\ restricted to the diagonal for a base change form. This is part of the project on factorisation of the base-change $p$-adic Asai $L$-function and these results will be made precise in the author's upcoming PhD thesis.

\subsection{Notations and conventions}\label{ssec:notations}

Let $p$ be a prime number. We work in the setup from \cite{LRZ}: $\GG$ is a connected reductive group over $\Q$ with a reductive model $G$ over $\ZZ_p$. In particular $G$ is quasi-split and splits over $\cO_K$ for an unramified extension $K$ of $\Q_p$ (see \cite[Corollary 5.2.14]{Con}), and denote $\cO_K$ by $\cO$. Furthermore, choose a Borel subgroup $B$ as well as a maximal torus $T\subset B$. Denote by $N\subset B$ the unipotent radical of $B$, by $A\subset T$ the maximal split torus in $T$ and by $Z$ the centre of $G$. Denote the Weyl group of $G$ by $W_G$. It is an \'etale group scheme over $\Z_p$ which is split over $\cO$. Denote its longest element by $w_0\in W_G(\Z_p)$ (it is defined over $\Z_p$ as $\Gal(K/\Q)$-action preserves the length).

\subsection*{Acknowledgments}
First and foremost, the author would like to express deep gratitude to David Loeffler and Sarah Zerbes for their constant encouragement and support throughout this project and for generously sharing their mathematical ideas as well as for reading through early versions of this text and suggesting numerous valuable improvements. Furthermore, the author is grateful to Antonio Cauchi, Kazım Büyükboduk and Andrew Graham for helpful discussions. The results of this note were presented on the 8-th Nisyros Conference on Automorphic Representations and Related Topics and the author would like to thank the participants and the organisers for their interest.

\numberwithin{theor}{section}
\section{Construction of the pairing}

\subsection{Sheaves on locally symmetric spaces}

Consider a level subgroup $U^p\subset \GG(\A_f^{(p)})$ with the property that $U^p\cdot U'$ is neat for any open compact $U'\subset G(\Z_p)$ (it suffices to check this condition for $U'=G(\ZZ_p)$ as a subgroup of a neat subgroup is neat). We also introduce the following subgroup following \cite[Remark 3.1.2]{LoeSph}:

\begin{defin}\label{E}
    Denote by $Z_0$ the maximal compact subgroup of $Z(\Q_p)$ and by $E$ the closure in $\GG(\Q_p)$ of the subgroup $Z(\Q)\cap (U^pZ_0)$.
\end{defin}

\begin{rem}
    The group $E$ is trivial if $G$ satisfies Milne's axiom SV5 (see \cite[Additional axioms, \S{5}]{MilneIntroSh}).
\end{rem}

Now choose some open $U\subset G(\Z_p)$ containing $E$ (which makes sense as $E$ is compact and central and hence contained in any maximal compact). Consider the associated locally symmetric space $$Y=Y_G(U^{p}\cdot U)\coloneqq \GG(\Q)\setminus \GG(\A)/U_{\infty}\cdot U^p\cdot U,$$
Where $U_{\infty}$ is the maximal compact modulo centre in $G(\R)$. 
We also consider the pro-covering $$\widetilde{Y}:=\varprojlim\limits_{E\subset U'\subset U} Y_G(U^pU)\to Y.$$ It is naturally a right $U/E$-torsor via right translations by elements of the latter group, we denote by $\pr\colon \widetilde{Y}\to Y$ the natural projection. 

A variant of Borel's construction gives a functor from the category of \mbox{(pro-)finite} $U/E$-sets to (pro-)finite coverings of $Y$, namely, equipping a constant sheaf on $\widetilde{Y}$ with $U/E$-action gives a descent data to $Y$. We explicitly define the functor and its notation as follows:

\begin{defin}
    Denote the functor $$S\mapsto \widetilde{Y}\times^{U/E}S\to Y$$ as $\Loc$.
\end{defin}

\begin{rem}\label{sheavesdefined}
    Note that if $G$ admits Shimura datum the corresponding Shimura varieties $Y_G(U^pU')$ and transition maps between them are defined over the reflex field $F$ (resp. over ${\cO_F}[S^{-1}]$ for Shimura datum of abelian type and $S$ large enough, see \cite[Theorem 2.2.1]{Lovering2017}) and hence the sheaves are defined over the corresponding bases. 
\end{rem}

We now work in the category of sets with $U$-action and specialise soon to the full subcategory of those with trivial $E$-action as we are ultimately interested in sheaves on $Y$. Moreover, we assume in what follows that $U$ has an Iwahori decomposition with respect to $B$. Denote the intersection of $U$ with $\overline{N}, T$ and $N$ by $\overline{N}_U, T_U$ and $N_U$ respectively. By assumption it follows that multiplication induces a bijection $\overline{N}_U\times T_U\times N_U\xrightarrow{\sim} U$.

\begin{rem}
    The results of this section may be generalised to an arbitrary parabolic instead of $B$.
\end{rem}   

Consider now an algebraic representation $V_{\lambda}$ of $G_K$ with the highest weight $\lambda$ and coefficients in $K$. Furthermore, choose a basis vector $f^{hw}_{\lambda}$ of the highest weight subspace of $V_{\lambda}$.

\begin{rem}
By Borel--Weil theorem one can realise isomorphism class of $V_{\lambda}$ as a subspace in $K[G]$ consisting of elements
\[
\{ f \in K[G] | f(\overline{n}\ell g) = \lambda(\ell) f(g) \ \forall \overline{n} \in \overline{N}, \ell \in T, g \in G \},
\]
equipped with $G$-action by right multiplication and one can take $f_\lambda^{\mathrm{hw}}$ whose restriction to the big Bruhat cell $\overline{N}TN$ is given by $\overline{n}\ell n \mapsto \lambda(\ell)$.
\end{rem}

Recall that an admissible lattice in $V_\lambda$ is an $\mathcal{O}$-lattice $\mathcal{L}_{\lambda} \subset V_\lambda$ which is invariant under $G_\mathcal{O}$ and whose intersection with the highest-weight subspace is $\mathcal{O} \cdot f_\lambda^{\mathrm{hw}}$. We pick an admissible lattice $\LL_{\lambda}$.

\begin{lemma}
    The orbit map given by $g\mapsto g\cdot f^{hw}_{\lambda}$ takes value in $\L_{\lambda}$ and factors through $U/N_U$.
\end{lemma}

\begin{proof}
    Indeed, as $U\subset G(\cO)$ it preserves $\LL_\lambda$ and as $f^{hw}_{\lambda}$ is a highest weight vector, it is $N$-invariant hence the assertion follows.
\end{proof}

Now consider the dual representation $V_{\lambda}^{\vee}$ and the dual lattice $\LL^{\vee}_{\lambda}$. The choice of $f_{\lambda}^{hw}$ determines the {\it lowest} vector $f^{lw}_{\lambda^{\vee}}\in \LL_{\lambda}^{\vee}\subset V_{\lambda}^{\vee}$ with the property that $(f^{hw}_{\lambda},f^{lw}_{\lambda^{\vee}})=1$. Repeating the argument we see that the orbit map $g\mapsto g\cdot f_{\lambda^{\vee}}^{lw}$ induces a map $U/\overline{N}_U\to \L_\lambda^\vee$.
\begin{defin}
Denote the induced maps by $\sigma_{\lambda} \colon U/N_U\to \LL_{\lambda}$, resp. $\overline{\sigma}_{\lambda^{\vee}}\colon U/\overline{N}_U\to \L_{\lambda}^{\vee}$.
\end{defin}
It follows from the Iwahori decomposition that $\overline{N}_U\backslash U/N_U\simeq T_U$.
\begin{defin}
Denote by $\pi\colon U/N_U\times U/\overline{N}_U\to T_U$ the map given by $([g],[h])\mapsto[h^{-1}g]\in \overline{N}_U\backslash U/N_U\simeq T_U$.
\end{defin}

This is a morphism of $U$-sets where $T_U$ is equipped with the trivial action.

\begin{defin}
    Denote by $ev$ the evaluation map $\LL_{\lambda}\times \LL^{\vee}_{\lambda}\to \cO$. 
\end{defin}
Much of our construction is based on the following elementary observation:
\begin{prop}\label{diagsetsprep}
    The diagram
    
\begin{tikzcd}
U/N_U\times  U/\overline{N}_U \arrow [r,"{\sigma_{\lambda},\overline{\sigma}_{\lambda^{\vee}}}"]\arrow[d,"\pi"]& \LL_{\lambda}\times \LL_{\lambda}^\vee \arrow[d,"ev"]\\
T_U \arrow [r,"\lambda"] & \cO
\end{tikzcd}

is commutative.
\end{prop}\label{sheafdiag}
\begin{proof}
    Consider the composition $$ev\circ(\sigma_{\lambda},\overline{\sigma}_{\lambda^{\vee}})\colon U/N_U\times U/\overline{N}_U\to\LL_{\lambda}\times \LL_{\lambda}^{\vee}\to \cO.$$ It is given by $$([g],[h])\mapsto (g\cdot f_{\lambda}^{hw}, h\cdot f_{\lambda^{\vee}}^{lw})=(h^{-1}gf^{hw}_{\lambda},f^{lw}_{\lambda^{\vee}}).$$ By the Iwahori decomposition for $U$ one can write $h^{-1}g=\bar{n}tn$ and using that the vectors are in highest (resp. lowest) weight subspaces the we see that $$(h^{-1}gf^{hw}_{\lambda},f^{lw}_{\lambda^{\vee}})=(tnf^{hw}_{\lambda},\bar{n}^{-1}f^{lw}_{\lambda^{\vee}})=(tf^{hw},f^{lw})=\lambda(t)\cdot (f^{hw}_{\lambda},f^{lw}_{\lambda^{\vee}})=\lambda(h^{-1}g),$$ hence the assertion follows.
\end{proof}

\begin{cor}\label{diagsetsmain}
Suppose $E$ acts trivially on $V_{\lambda}$. Then it acts trivially on $V^{\vee}_{\lambda}$ and the diagram from Proposition \ref{diagsetsprep} induces a commutative diagram

\begin{tikzcd}
U/(EN_U)\times  U/(E\overline{N}_U) \arrow [r,"{\sigma_{\lambda},\overline{\sigma}_{\lambda^{\vee}}}"]\arrow[d,"\pi"]& \LL_{\lambda}\times \LL_{\lambda}^\vee \arrow[d,"ev"]\\
T_U/E \arrow [r,"\lambda"] & \cO.
\end{tikzcd}

\end{cor}

\begin{rem}
    One can also prove similar results replacing $B$ by a parabolic $Q$ and the torus $T$ by the abelianisation of its Levi.
\end{rem}

\begin{rem}
    The right multiplication $r_t\colon v \mapsto vt$ by $T_U$ on $U$ descends to the sets $U/N_U,~U/EN_U,~U/E\overline{N}_U,~U/E\overline{N}_U$ as $T_U$ normalises the groups $N_U,~EN_U,~E\overline{N}_U$ and $\overline{N}_U$. Moreover, it commutes with left multiplication by $U$ giving the action of diamond operators on the pro-coverings $\Loc(U/EN_U)$ and $\Loc(U/E\overline{N}_U)$ and hence of the Iwasawa algebra on their cohomology.
\end{rem}
We record the following observation whose proof is straightforward from the definitions:
\begin{lemma}\label{torusequiv}
The diagram in Proposition~\ref{diagsetsprep} is $T_U\times T_U$--equivariant with respect to the action of $(s,t)\in T_U\times T_U$ given by $r_s\times r_t$ on $U/N_U\times U/\overline{N}_U$, $r_{st^{-1}}$ on $T_U$, $(\lambda(s)\times \lambda(t^{-1}))$ on $\LL_{\lambda}\times \LL^{\vee}_{\lambda}$ and $\lambda(st^{-1})$ on~$\cO$.
\end{lemma}

We would like to relate the two towers of locally symmetric spaces over $Y$ given by $\Loc(U/EN_U)$ and $\Loc(U/E\overline{N}_U)$. In the case of modular curves and $U=\Gamma_0(p^n)$ or $\Gamma_1(p^n)$ these towers are related by the Atkin--Lehner involution. Although for general $U$ there is no direct analogue, it turns out that certain choices of $U$ do yield one. 

\subsection{Tits representative of the longest element}
As $B_\cO$ and $T_\cO$ are stable under $\Gal_{K/\Q_p}$ the latter group naturally acts on the roots of $G_\cO$ preserving positive roots.
As $G_{\cO}$ is split we can choose a pinning for this group i.e. a trivialisation of root subgroups for all simple roots (\cite[Theorem~4.1.4]{Con}). Moreover, as $\Gal_{K/\Q_p}$ permutes simple roots we can choose a pinning which is stable under this action. 
This gives rise to a Galois-equivariant set-theoretic section of $N_{G}(T)(\cO)\to W_G(\cO)$. In general the section is not a group homomorphism and it is constructed vie reduced expressions of elements of $W_G(\cO)$. Its image contained in the Tits group (see \cite[Corollary~5.1.11, Example~6.4.3]{Con}  and \cite{Tits} or \cite[Section~5]{AdVog} for fields).

\begin{defin}
    Denote by $s\colon  W_G(\cO)\to N_{G}(T)(\cO)$ the section associated to our choice of the pinning.
\end{defin}

The following property of $s(w_0)$ will be used to construct the level subgroups.

\begin{lemma}
    The element $s(w_0)$ lies in $G(\Z_p)$ and $s(w_0)^2$ is central in $G$.
\end{lemma}
\begin{proof}
    Observe that the longest element $w_0\in W(G_\cO)$ is stable under Galois action, therefore its lift $s(w_0)$ is also Galois-invariant, thus $s(w_0)\in G(\ZZ_p)$.
    For the second claim, this can be checked over $\Q_p$ where it follows from \cite[Lemma 5.4]{AdVog}.
\end{proof}

\begin{defin}
    We denote $s(w_0)$ by $w$.
\end{defin}

\subsection{Level subgroups}\label{lvl} As in \cite{LRZ} choose a strictly dominant cocharacter $\eta$ of $A$, but subject to the extra condition that $\eta+w_0(\eta)$ factors through $Z(G)$: this can be achieved by picking an arbitrary strictly dominant $\eta'$ and setting $\eta\coloneqq \eta'-w_0(\eta')$. This particular choice might not be optimal in the sense that it might not provide a minimal level subgroup in what follows, however, they all have the same intersection with $B$ and therefore provide the same anti-ordinary part of the cohomology of the associated locally symmetric space with coefficients in $\Loc(\L_{\lambda})$ by \cite[Remark 2.23]{LRZ}.

Set $\tau=\eta(p)$, $N_0=N(\ZZ_p)$, $\overline{N}_0=\overline{N}(\Z_p)$ and $\overline{N}_1=\tau^{-1}\overline{N_0}\tau$, also set $T_0=T(\Z_p)$. Consider the level subgroup $U=\overline{N}_1T_0N_0$ (see \cite[Section 4.4]{LoeSph} where similar subgroups are considered, one checks similarly that $U$ is indeed a subgroup). 

\begin{lemma}\label{centre}
    We have $(w\tau)^2\in Z(G)(\Q_p)$.
\end{lemma}
\begin{proof}
Indeed, $w\tau w\tau=w^2(w^{-1}\tau w\tau)=w^2\cdot(w_0(\eta)+\eta)(p)$ and both multiples are central.
\end{proof}

\begin{lemma}\label{uniconj}
    We have $w\tau N_0(w\tau)^{-1}=\overline{N}_1$ and $w\tau \overline{N}_1(w\tau)^{-1}=N_0$.
\end{lemma}
\begin{proof}
    By definition $\overline{N}_1=\tau^{-1}\overline{N_0}\tau$ and we have $\overline{N}_0=w^{-1}N_0w$ so we need to check that $w\tau N_0(w\tau)^{-1}=\tau^{-1}w^{-1}N_0w\tau$ i.e. that $(w\tau)^2N_0(w\tau)^{-2}=N_0$ which follows from Lemma~\ref{centre}. The second assertion follows from the first and Lemma~\ref{centre}.
\end{proof}

\begin{cor}\label{norm}
   The subgroup $U$ is normalised by $w\tau$.
\end{cor}
\begin{proof}
    Indeed, using Lemma \ref{uniconj} we see that $$w\tau U(w\tau)^{-1}=w\tau\overline{N}_1(w\tau)^{-1}w\tau T_0(w\tau)^{-1} w\tau N_0(w\tau)^{-1}=N_0T_0\overline{N}_1=U,$$
    So the claim follows.
\end{proof}

\subsection{AL-operator}\label{A-Lsec}

It follows from Corollary \ref{norm} that the right translation $R_{{(w\tau)}^{-1}}$ by the element $(w\tau)^{-1}$ descends to an automorphism of $Y$, which we denote by $\AL$; one has the following commutative diagram.
\[
\begin{tikzcd}
  \widetilde Y \arrow[r, "R_{(w\tau)^{-1}}"] \arrow[d, "\pr"'] & \widetilde Y \arrow[d, "\pr"] \\
  Y \arrow[r, "\AL"'] & Y.
\end{tikzcd}
\]
In particular, $R_{(w\tau)^{-1}}$ defines an isomorphism between $\AL^*(\widetilde{Y})$ and $\widetilde{Y}$ as spaces over $Y$ but with $w\tau$-twisted right torsor structure. 

Consider the involution $(-)^{w\tau}$ on the category of  $U/E$-sets given by conjugating the action by $w\tau$.

\begin{lemma}\label{ALconj}
    There is a functorial isomorphism $\Loc(S^{w\tau})\simeq \AL^*(\Loc (S))$.
\end{lemma}
\begin{proof}
  Indeed, we compute using the observation above:
  $$\AL^*(\widetilde{Y}\times^{U/E}S)\simeq \widetilde{Y}\times^{U/E,w\tau}S\simeq \widetilde{Y}\times^{U/E}S^{w\tau}.$$
\end{proof}

\begin{lemma}
    We have $(U/N_0)^{w\tau}\simeq U/\overline{N}_1$, and $(U/EN_0)^{w\tau}\simeq U/E\overline{N}_1$.
\end{lemma}
\begin{proof}
    The first isomorphism follows from Lemma \ref{uniconj} and the second follows from the first as $E$ is central.
\end{proof}

We consider dominant weights $\lambda$ with the property that $\lambda|_{E}$ is trivial, so representations with such highest weight $\lambda$ give rise to sheaves on $Y$ via $\Loc$. Let $\lambda$ be such a dominant weight. Applying $w\tau$-conjugation to $\sigma=\sigma_{\lambda} \colon U/N_0\to \LL_{\lambda}$ we get a section $\sigma^{w\tau}\colon U/\overline{N}_1\to (\LL_{\lambda})^{w\tau}$. Let us identify $(V_{\lambda})^{w\tau}$ with the same vector space as $V_{\lambda}$ but with the $w\tau$--twisted action of $G$ and similarly for $(\LL_{\lambda})^{w\tau}$. Observe that there is a $G$-equivariant morphism $(V_{\lambda})^{w\tau}\to V_{\lambda}$ which is given by $(w\tau)^{-1}$. Since it does not preserve integral lattices we introduce a certain normalisation as in \cite{LRZ}: denote by $\langle\cdot,\cdot\rangle$ the natural pairing between characters and cocharacters of $T$.

\begin{lemma}\label{normalise}
    The morphism $p^{\langle\eta,\lambda\rangle}(w\tau)^{-1}$ sends $(\LL_{\lambda})^{w\tau}$ to $\LL_{\lambda}$.
\end{lemma}

\begin{proof}
    As $w\in G(\ZZ_p)$ this follows directly from \cite[Lemma 2.3.4]{LRZ}.
\end{proof}

In what follows we use the notion of (compactly supported)  cohomology for $Y$ in the same manner as \cite[Section~2.4]{LRZ}, i.e.~it can be referred to Betti cohomology of the corresponding smooth manifold, geometric \'etale cohomology of canonical model or absolute \'etale cohomology of the integral model (when the latter two make sense). By Remark \ref{sheavesdefined} the corresponding cohomology with coefficients in the sheaves considered hereafter also makes sense. 
For convenience we omit the functor $\Loc$ from notation when the context is clear.

\begin{defin}\label{Atk-Leh}
The integrally normalised Atkin--Lehner operator $\AL_{\lambda}$ on the cohomology of $Y$ with coefficients in $\LL_{\lambda}$ is defined as the composition 

$$H^{\bullet}(Y,\LL_\lambda)\to H^{\bullet}(Y,\AL^*\LL_\lambda)\to H^{\bullet}(Y,\LL_\lambda),$$

where the second arrow is induced by $p^{\langle\eta,\lambda\rangle}(w\tau)^{-1}$ and the identification $\Loc(\AL^*\LL_\lambda)=\Loc(\LL_\lambda^{w\tau})$ given by Lemma \ref{ALconj}.
\end{defin}

\begin{rem}
    This is well-defined by Lemma \ref{normalise}. For modular curves of Iwahori level it is (up to scalar) the Atkin--Lehner operator usually denoted by $W$. Note that similarly $\AL^2$ is an endomorphism which commutes with all Hecke operators. To avoid confusion with the Weyl group we chose a different notation. Sometimes we omit the weight subscript and just write $\AL$.
\end{rem}

The automorphism $\AL$ also provides a map between the Iwasawa cohomology groups associated to $EN_0$ and $EN_1$ defined as in \cite[Definition~2.6.1]{LRZ} as well as between compactly supported versions thereof. Recall the definition:

\begin{defin}\label{Iwasawacoh}
    Let $Q\subset U$ be a closed subgroup containing $E$. Define $$H^{\bullet}_{Iw,(c)}(Q)\coloneqq \varprojlim\limits_{U'\supset Q}H_{(c)}^{\bullet}(Y(U^pU'),\Z_p),$$ where the limit is taken with respect to the trace map. (Here and in what follows the subscript $c$ stands for compactly supported cohomology.)
\end{defin}

As $N_0$ and $\overline{N}_1$ are conjugated under $w\tau$ we have an isomorphism $$\varprojlim\limits_{U\supset E\overline{N}_1} Y(U^pU)\simeq \AL^*( \varprojlim\limits_{U\supset EN_0} Y(U^pU))$$ of profinite coverings of $Y$ provided by Lemma~\ref{ALconj} yielding an isomorphism $$\AL^*\colon H^{\bullet}_{Iw}(EN_0)\simeq H^{\bullet}_{Iw}(E\overline{N}_1).$$

\begin{rem}\label{ALIwasawa}
    We know from Lemma \ref{torusequiv} that there is a $T_0$-action on the Iwasawa cohomology groups above. It follows from the definitions that $\AL^*$ intertwines the action of $T_0$ on $H^{\bullet}_{Iw}(EN_0)$ with the action of $T_0$ conjugated by $w_0$ on $H^{\bullet}_{Iw}(E\overline{N}_1)$.
\end{rem}

Consider the composition $$\sigma'=p^{\langle\eta,\lambda\rangle}(w\tau)^{-1}\circ\sigma^{w\tau}\colon U/E\overline{N}_1\to (\LL_{\lambda})^{w\tau}\to \LL_{\lambda},$$
it is given by
\begin{equation}\label{ALorbit}
    [g]\mapsto g\cdot p^{\langle\eta,\lambda\rangle}(w\tau)^{-1}f^{hw}_{\lambda}=p^{\langle \eta,\lambda-w_0\lambda\rangle}g\cdot w^{-1} f_{\lambda}^{hw}.
\end{equation}

We recall the following concept defined in \cite[Definition 5.1.1]{LRZ} based on \cite[Sec 12.2.2] {Kings2015}.
\begin{defin}\label{moment}
    Denote by  $$\mom^{\lambda}\colon H^{\bullet}_{Iw,(c)}(EN_0)\to H^{\bullet}_{(c)}(Y,\LL_{\lambda}), \; \widetilde{\mom}^{\lambda} H^{\bullet}_{Iw,(c)}(E\overline{N}_1)\to H^{\bullet}_{(c)}(Y,\LL_{\lambda})$$ 
    the maps induced by $\sigma$, (resp $\sigma'$) from the corresponding (compactly supported) Iwasawa cohomology groups. They are called moment maps.
\end{defin}

\begin{rem}
    As the map $\sigma$ depends on the choice of a highest weight vector $f^{hw}_{\lambda}$ the moment maps depend on this choice as well.
\end{rem}

\begin{lemma}\label{Atkmom}
   The diagram
    
\begin{tikzcd}
H^{\bullet}_{Iw}(EN_0) \arrow [r,"\AL^*"]\arrow[d,"\mom^{\lambda}"]& H^{\bullet}_{Iw}(E\overline{N}_1) \arrow[d,"\widetilde{\mom}^{\lambda}"]\\
H^{\bullet}(Y,\LL_{\lambda}) \arrow [r,"\AL_{\lambda}"] & H^{\bullet}(Y,\LL_{\lambda})
\end{tikzcd}

is commutative.
\end{lemma}
\begin{proof}
    By definition, $\sigma'$ is the composition of $\sigma^{w\tau}$ and $p^{\langle\eta,\lambda\rangle}(w\tau)^{-1}$. So $\widetilde{\mom}^{\lambda}$ is the composition of $\AL^*(\mom^{\lambda})$ and $p^{\langle\eta,\lambda\rangle}(w\tau)^{-1}$ i.e. we have the following diagram with a commutative upper square:
    
\begin{tikzcd}
H^{\bullet}_{Iw}(EN_0) \arrow [r,"\AL^*"]\arrow[d,"\mom^{\lambda}"]& H^{\bullet}_{Iw}(E\overline{N}_1) \arrow[d,"\AL^*({\mom}^{\lambda})"]\\
H^{\bullet}(Y,\LL_{\lambda}) \arrow [r,"\AL^*"] & H^{\bullet}(Y,\AL^*(\LL_{\lambda}))\arrow[d,"p^{\langle\eta,\lambda\rangle}(w\tau)^{-1}"]\\
& H^{\bullet}(Y,\LL_{\lambda}).
\end{tikzcd}\newline
The claim follows as the composition of the lowest arrow with $\AL^*$ is $\AL_{\lambda}$. 
\end{proof}

\begin{rem}\label{changelevel} More generally, the results stay the same if we change the level subgroup $U$ to another one admitting Iwahori decomposition and stable under conjugation by $w\tau$. Furthermore, one can replace $N_0$ and $\overline{N_1}$ by a subgroup $N_0'$ such that $T_UN_U\supset N_0'\supset N_0$ and its $w\tau$-conjugate subgroup if one considers weights which are trivial on $N'_0/N_0$
\end{rem}

\subsection{Duality for the Iwasawa cohomology}
As in Definition \ref{moment} denote by $\overline{\mom}^{\lambda^{\vee}}$ the moment map induced by $\overline\sigma_{\lambda^{\vee}}$. We also denote by $\underline{X}$ the constant sheaf associated to the set $X$. The character $\lambda\colon T_0/E\to \cO^\times$ induces an $\cO$-linear map $\cO[[T_0/E]]\to \cO$ and we denote analogously by $\mom^{\lambda}$ the corresponding map on cohomology with values in corresponding constant sheaves.
\begin{prop}\label{dualityprep} Let $d$ be an arbitrary integer. The following diagram is commutative:

\begin{tikzcd}[column sep=huge]
H^{\bullet}_{Iw,c}(EN_0)\otimes  H^{d-\bullet}_{Iw}(E\overline{N}_1) \arrow [r,"{{\mom}^{\lambda},{\overline{\mom}}^{{\lambda}^\vee}}"]\arrow[d,"\pi_*"]& H_c^{\bullet}(Y,\LL_{\lambda})\otimes H^{d-\bullet}(Y,\LL_{\lambda}^{\vee}) \arrow[d,"{(\cdot,\cdot)}"]\\
H_c^d(Y,\underline{\cO[[T_0/E]]}) \arrow[r,"{\mom}^\lambda"]& H_c^d(Y,\cO).
\end{tikzcd}

\end{prop}
\begin{proof}
  It follows via taking cohomology of the sheaves in the commutative diagram from Corollary  \ref{diagsetsmain}.
\end{proof}

\begin{rem}
    One has $H_{c}^d(Y,\underline{\cO[[T_0/E]]})=\cO[[T_0/E]]\otimes H_{c}^d(Y,\cO)$. The result is valid for arbitrary $d$, however, for potential applications most interesting cases seem to be $d=\dim(Y)$ where we take the real dimension of the locally symmetric space and consider geometric \'etale or Betti cohomology (in which case $H_{c}^d(Y,\cO)$ is free of rank $1$), or $d=2\dim(Y)+1$ if $Y$ is an $\cO_F[S^{-1}]$ model of a Shimura variety of abelian type, $\dim(Y)$ is its relative dimension over $\cO_F[S^{-1}]$ and we consider its absolute \'etale cohomology.
\end{rem}

Comparing Iwasawa cohomology of $EN_0$ and $E\overline{N_1}$ via the Atkin--Lehner map we obtain our main result. To fix normalisations, we set $\mom^{\lambda^{\vee}}$ to be the moment map associated to the choice of the highest vector $f^{hw}_{\lambda^{\vee}}\in \L_{\lambda}^{\vee}$ with the property that $(wf_{\lambda}^{hw},f_{\lambda^{\vee}}^{hw})=1$. Set
$$c_{\lambda}:=p^{\langle \eta,\lambda-w_0\lambda\rangle}.$$
Recall that the highest weight $\lambda^{\vee}$ of $V_{\lambda}^{\vee}$ is equal to $-w_0\lambda$, so $c_{\lambda}=c_{\lambda^{\vee}}$.
\begin{theor}\label{dualitymain}
There is a commutative diagram

\begin{tikzcd}[column sep=huge]
H^{\bullet}_{Iw,c}(EN_0)\otimes  H^{d-\bullet}_{Iw}(EN_0) \arrow [r,"{{\mom}^{\lambda},{\mom}^{\lambda^{\vee}}}"]\arrow[d,"d_{Iw}"]& H_c^{\bullet}(Y,\LL_{\lambda})\otimes H^{d-\bullet}(Y,\LL_{\lambda}^{\vee}) \arrow[d,"{(\cdot,\AL\cdot)}"]\\
H_c^d(Y,\underline{\cO[[T_0/E]]}) \arrow[r,"c_{\lambda}\cdot\mom^{\lambda}"]& H_c^d(Y,\cO).
\end{tikzcd}

\end{theor}

\begin{proof} As the highest weight $\lambda^{\vee}$ of $V_{\lambda}^{\vee}$ is equal to $-w_0\lambda$, it follows from (\ref{ALorbit}) applied to $\lambda^{\vee}$ that with our choice of normalisation we have  ${\widetilde{\mom}^{\lambda^\vee}=c_{\lambda}\cdot\overline{\mom}^{\lambda^{\vee}}}$. Now, by Proposition \ref{dualityprep} the diagram

\begin{tikzcd}[column sep=huge]
H^{\bullet}_{Iw,c}(EN_0)\otimes  H^{d-\bullet}_{Iw}(E\overline {N}_1) \arrow [r,"{{\mom}^{\lambda},\widetilde{\mom}^{\lambda^{\vee}}}"]\arrow[d,"\pi"]& H_c^{\bullet}(Y,\LL_{\lambda})\otimes H^{d-\bullet}(Y,\LL_{\lambda}^\vee) \arrow[d,"{(\cdot,\cdot)}"]\\
H_c^d(Y,\underline{\cO[[T_0/E]]}) \arrow[r,"c_{\lambda}\cdot\mom^{\lambda}"]& H_c^d(Y,\cO)
\end{tikzcd}\newline
commutes. Hence it follows from Lemma \ref{Atkmom} (applied to $\lambda^{\vee}$) that if we define $d_{Iw}$ to be the composition of $\pi$ and $\Id\otimes \AL^*$ the resulting diagram is commutative.
\end{proof}

\begin{rem}\label{Iwasawaequiv}
The involution $t\mapsto w_0(t^{-1})$ of $T$ determines an automorphism of $\cO[[T_0/E]]$. It follows from Lemma ~\ref{torusequiv} and Remark~\ref{ALIwasawa} and that if we consider the action of $\cO[[T_0/E]]$ twisted by this involution on $H^{d-\bullet}_{Iw}(E\overline {N}_1)$, the left vertical arrow becomes $\cO[[T_0/E]]$-bilinear. 

Furthermore, as in Remark \ref{changelevel} the statements remain true if we replace $N_0$ and $\overline{N}_1$ by $N_0'$ and its $w\tau$-conjugate subject to considering weights $\lambda$, such that $\lambda$ and $-w_0\lambda$ are trivial on $N'_0/N_0$.

One can also adapt the construction to work in a more general setting allowing Iwasawa cohomology associated to the unipotent radical of a parabolic whose Levi stable under $w$-conjugation.
\end{rem}

\subsection{Duality for anti-ordinary cohomology}

Although our result is stated for the whole cohomology groups involved, the main application we have in mind is concerned with anti-ordinary parts thereof so here we briefly explain why one can take anti-ordinary parts of cohomology groups in Theorem \ref{dualitymain} and the statement remains valid. The (anti)-ordinary subgroup of $H_{(c)}^{\bullet}(Y,\LL_\lambda)$ is the maximal subgroup on which the optimally integrally normalised Hecke operator associated to the double coset $[U\eta(p)U]$ (resp. $[U\eta(p^{-1})U]$) acts invertibly. Recall the formal definition (the first part is \cite[Definition 2.9]{LRZ}):

\begin{defin}\hspace{0cm}\begin{enumerate}
    \item[(i)] For any strictly dominant $\eta_0$ define $\mathcal{T}'_{\eta_0}$ to be $p^{\langle \eta_0,\lambda\rangle}[U\eta_{0}(p^{-1})U]$.
    \item [(ii)] For any strictly dominant $\eta_0$ define $\mathcal{T}_{\eta_0}$ to be $p^{\langle \eta_0,-w_0\lambda\rangle}[U\eta_{0}(p)U]$.
\end{enumerate}
\end{defin}
 These operators are defined on $H^{\bullet}_{(c)}(Y,\L_{\lambda})$ and one defines the (anti)-ordinary projector $e$ (resp. $e'$) to be $\lim \mathcal{T}_{\eta_0}^{n!}$ (resp. $\lim (\mathcal{T}_{\eta_0}')^{n!}$), these do not depend on the choice of $\eta_0$ (see \cite[Proposition 2.12]{LRZ} for details; the proof for $e$ is identical).

\begin{lemma} We have
    $e'\circ\AL=\AL\circ e$ and $e\circ\AL=\AL\circ e'$.
\end{lemma}
\begin{proof}
One has
$$(w\tau)^{-1}\cdot\eta_0(p^{-1})\cdot w\tau=(w_0\eta_0)(p^{-1})=(-w_0(\eta_0))(p).$$

Since $w\tau$ normalises $U$ we deduce that $[U\eta_0(p^{-1})U]\circ \AL$ is the Hecke operator associated to the double coset $[U\eta_0(p)w\tau U]$ and $\AL\circ [U(-w_0(\eta_0)(p))U]$ is the Hecke operator associated to the double coset $[Uw\tau(-w_0(\eta_0))(p)U]$ which is the same double coset. Furthermore, as $\langle\eta_0,\lambda\rangle=\langle-w_0\eta_0,-w_0\lambda\rangle$ we deduce that $\mathcal{T}'_{\eta_0}\circ \AL=\AL\circ\mathcal{T}'_{-w_0\eta_0}$, so the first claim follows as $-w_0\eta_0$ is strictly dominant and $e$ does not depend on the choice of the strictly dominant cocharacter. The second claim follows similarly.
\end{proof}

\begin{cor}\label{dualityord}
The diagram from Theorem~\ref{dualitymain} induces a commutative diagram

\begin{tikzcd}[column sep=huge]
e'H^{\bullet}_{Iw,c}(EN_0)\otimes  e'H^{d-\bullet}_{Iw}(EN_0) \arrow [r,"{{\mom}^{\lambda},{\mom}^{\lambda^{\vee}}}"]\arrow[d,"d_{Iw}"]& e'H_c^{\bullet}(Y,\LL_{\lambda})\otimes e'H^{d-\bullet}(Y,\LL_{\lambda}^{\vee}) \arrow[d,"{(\cdot,\AL\cdot)}"]\\
H_c^d(Y,\underline{\cO[[T_0/E]]}) \arrow[r,"c_{\lambda}\cdot\mom^{\lambda}"]& H_c^d(Y,\cO).
\end{tikzcd}

\end{cor}

\section{Examples and applications}\label{exappl}
In this section we briefly outline how the pairing constructed in \cite{OhtES} (or rather a variant from \cite{Ohta2003}) for modular curves relates to constructions. Furthermore, we consider an example of $p$-adic Hirzebruch--Zagier cycles from \cite{CaNiRo} and show that (after suitable projection to the cuspidal part) the intersection index of their specialisations at even parallel weights against their Atkin--Lehner images interpolates $p$-adically. The latter example is motivated by the discussion after \cite[Definition~7.10]{CaNiRo}.
\subsection{Relation to Ohta's pairing}
We now show how the results above recover the pairing of Ohta for $G=\GL_2$ (so we are in Milne's SV5 setting and $E$ is trivial). For this subsection denote $\Lambda:=\ZZ_p[[\ZZ_p^{\times}]]$ and set $U^{p}=U_1(M)$. Denote by $Y_r$ the modular curve of level $\Gamma_1(M)\cap\Gamma_1(p^{r})$ and by $Y$ the modular curve of level $\Gamma_1(M)\cap\Gamma_0(p)$. Ohta considers the \'etale cohomology of the tower $\{Y_{r,\overline{\Q}}\}$.

Set $\bH_{(c)}\coloneqq \varprojlim H^1_{\acute{e}t,(c)}(Y_{r,\overline{\Q}},\Z_p)$. The highest weight vector $x^{\otimes k}$ defines the map $$\mom^k\colon \bH_{(c)}\to H^1(Y(\Gamma_1(M)\cap\Gamma_0(p),\TSym^k(\Z_px\oplus \Z_py)).$$

We can take $ \tau= \begin{bmatrix}
    p & 0 \\
    0 & 1
\end{bmatrix},$ this choice yields $U=U_0(p)=\begin{bmatrix}
    * & * \\
    0 & *
\end{bmatrix} \pmod p$.

The Shimura variety of level 
$\begin{bmatrix}
    1 & * \\
    0 & 1
\end{bmatrix} \pmod {p^r}$ for $G$ is $Y_r\times_{\Q}\Q(\mu_{p^r})$, hence taking the inverse limit we obtain $H^1_{Iw,(c)}(N_0)\simeq \bH_{(c)}\hat{\otimes}\Lambda$. One can verify the following:

\begin{lemma}\label{action}
    Representing $T$ as the direct product of $T^0=T\cap SL_2$ and $T^1=\begin{bmatrix}
    * & 0 \\
    0 & 1
\end{bmatrix}$ under the above isomorphism the action of $(s,t)\in (T^0\times T^1)(\ZZ_p)$ is given by $\langle s\rangle\otimes [t]$.
\end{lemma}

\begin{prop}
 Restricting $d_{Iw}$ to the subspaces $\bH_{(c)}\otimes1_{\Lambda}\subset H^{1}_{Iw,c}(N_0)$ and taking $\lambda$ of the form $[k,0]$ we obtain a $\Lambda$--bilinear pairing between $\bH$ and $\bH_{c}$ valued in a free $\Lambda$-module of rank $1$. Taking ordinary parts recovers results from \cite[Theorem~1.3.3]{Ohta2003} for the weights of $\Lambda$ which are integers (i.e. arithmetic weights with trivial Dirichlet character of $p$-power conductor part).
\end{prop}

\begin{proof}
A weight $\lambda$ of $T$ is given by a pair $(k,l)$ of weights of $T^{0}, T^1$. Restricting to $\lambda=(k,0)$ we see that $c_\lambda\cdot \mom^{\lambda}$ factors through $$H_c^2(Y,\underline{\Z_p[[T_0]]})\to H_c^2(Y,\underline{\Z_p[[T^0(\Z_p)]]})\simeq H_c^2(Y,\underline {\Lambda}).$$

Moreover, $-w_0(k,0)=(k,-k)$ so choosing $$f^{hw}_{[k,0]}=x^{\otimes k}\in \TSym^k(\ZZ_px\oplus\ZZ_py)$$ and restricting $\mom^{\lambda}$ and $\mom^{\lambda^{\vee}}$ to the subspaces $\bH\simeq \bH\otimes1_{\Lambda}\subset H^{1}_{Iw,c}(N_0) $ which are $\ZZ_p[[T^0(\Z_p)]]$--stable by Lemma \ref{action} we see that these restrictions are given by $\mom^k$ (resp $(-1)^k\mom^k$ and as the involution from Remark \ref{Iwasawaequiv} is identity on $T^0$ we conclude that $d_{Iw}$ restricted to $\bH\otimes \bH_c$ is $\Lambda$--bilinear. 

The dominant weight is $\eta=(0,1)$ in the chosen coordinates and its inner product with both $(1,0)$ and $(0,1)$ is equal to $1$.
Therefore,  $c_{[k,0]}=p^{k}$ and the integrally normalised $\AL$ is $p^k$ times the rational Atkin--Lehner operator $W_{Mp}$ up to Hecke operators outside of $p$ which are invertible and commute with trace and moment maps so we obtain the claim. 
\end{proof}
Note that these points form a dense subset in the weight space so it would be interesting to recover formulas for specialisations at other arithmetical weights which are given in \cite[Theorem~1.3.3]{Ohta2003}.

\begin{rem}\label{Ohtarem}\hspace{0cm}
    \begin{enumerate}
        \item [(i)] In \cite{OhtES} Ohta considers the modular curve of level $\Gamma_1(Np)$ instead of $\Gamma_0(p)\cap \Gamma_1(N)$ as the base. However, due to Remark \ref{changelevel} one gets similar results for $\Gamma_1(Np)$.
        \item [(ii)]  As already mentioned, in \cite{OhtES} the duality is stated using operator $W_{Mp}$ whereas we consider $\AL$ but they differ by $\langle p\rangle_M W_M$ which is invertible and commutes with the trace maps defining $\bH_{(c)}$. Moreover, the moment map is equivariant with respect to  $\langle p\rangle_MW_M$ So recovering the duality for $W_{Mp}$ amounts to composing $d_{Iw}$ with $\Id\otimes \langle p\rangle_M W_M$.
        \item [(iii)] As a starting input in the construction of his pairing Ohta also considers a pairing between pro-\'etale sheaves on the curves $Y_1(Np^r)_{\overline{\Q}}$ which correspond to the quotients of the level $p$ subgroups by the upper-triangular unipotent part (see \textsection{1} of loc.~cit.).
    \end{enumerate}
\end{rem}

\subsection{The case of $p$-adic Hirzebruch--Zagier cycle}
We consider now a quadratic extension of totally real fields $L/F$ with $L$ unramified at $p$, denote $[F:\Q]=d$. Consider $\GG=\Res_{L/\Q}\GL_2$ (one can also work with $\GG^*=\GG\times_{\Res_{L/\Q}\mathbb{G}_m}\mathbb{G}_m$). Similarly, choosing $ \tau= \begin{bmatrix}    p & 0 \\    0 & 1\end{bmatrix}$
yields $U=U_0(p)$; take $U^p=U_1(\NN)$. In \cite{CaNiRo} the authors construct an Iwasawa cohomology class $\mathscr{Z}_{\infty}\in H^{2d}_{Iw}(EQ_G)$ related to Hirzebruch--Zagier cycle (here again the cohomology considered is the \'etale cohomology of the base-change to $\overline{\Q}$). The closed subgroup $Q_G$ is defined as the product of the upper-triangular unipotent $N$ and $\Res_{{\cO_L}/\ZZ}\mathbb{G}_m$ embedded into the diagonal maximal torus via $ \alpha \mapsto \begin{pmatrix}    \alpha & 0 \\    0 & \bar\alpha\end{pmatrix}$, where $z\mapsto \bar{z}$ is the non-trivial element of $\Gal(L/F)$.

Denote by $\Lambda$ the Iwasawa algebra from \cite[Section~4.4]{CaNiRo} . By Proposition \ref{dualitymain} and Remark \ref{Iwasawaequiv} for any $\mathscr{Z}'_{\infty}\in H^{2d}_{Iw,c}(EQ_G)$ we get a function valued in the $1$-dimensional $\Lambda$-module $H^{2d}_c(Y,\Lambda)$ interpolating the intersection pairing of $\mom^{\lambda}(\mathscr{Z}_{\infty})$ and $\AL(\mom^{\lambda^{\vee}}(\mathscr{Z}'_{\infty}))$ normalised by $c_{\lambda}$ for the weights $\lambda$ trivial on $Q_G$ (and $E$).

Furthermore, as in \cite[Section~2.2]{CaNiRo} under the identification of dominant weights trivial on the embedded $\Res_{{\cO_L}/\ZZ}\mathbb{G}_m$ with $(\underline{k},m) \in (\mathbb{Z}_{\ge 0})^{\Sigma_L} \times \mathbb{Z}$
such that $k_\sigma \equiv m \pmod{2}, k_{\sigma}=k_{\bar{\sigma}}$ where
$\underline{k} = (k_\sigma)_{\sigma \in \Sigma_L}$
and taking lattices in the corresponding representations of the form
$$\LL^{\underline{k},m}
:= \bigotimes_{\sigma \in \Sigma_L}
\TSym^{k_\sigma}(\cO e_{1,\sigma}\oplus \cO e_{2,\sigma})
\otimes \det{}_{\sigma} ^{\frac{m-k_\sigma}{2}}$$

one has $-w_0(\underline{k},m)=(\underline{k},-m)$ under the standard duality and $c_{(\underline{k},m)}=p^{\sum k_{\sigma}}$.  

Moreover, one can take $w := \begin{pmatrix}    0 & -1 \\    1 & 0\end{pmatrix}$ and taking the highest weight vector
$$f^{hw}_{(\underline{k},m)}=\bigotimes_{\Sigma_L} (e_{1,\sigma}^{\otimes k_{\sigma}}\otimes(e_{1,\sigma}\otimes e_{2,\sigma}-e_{2,\sigma}\otimes e_{1,\sigma})^{\otimes \frac{m-k_{\sigma}}{2}})$$ one obtains the following:

\begin{lemma}\label{selfdual}
    For the (self-dual) weights $$(\underline{k},0) \in (\mathbb{Z}_{\ge 0})^{\Sigma_L} \times \mathbb{Z},$$
such that $k_\sigma \equiv 0 \pmod{2},~k_{\sigma}=k_{\bar{\sigma}}$ one has $\LL^{\underline{k},0}\subset {\LL^{\underline{k},0}}^{\vee}$ and the vectors $f^{hw}_{(\underline{k},0)}$ satisfy $(f^{hw}_{(\underline{k},0)},wf^{hw}_{(\underline{k},0)})=1$.
\end{lemma}

In particular for such weights the map $\mom^{\lambda^{\vee}}$ is the composition of the natural embedding with $\mom^{\lambda}$.

While the results of \cite{CaNiRo} and \cite{LRZ} provide classes in the non-compactly supported Iwasawa cohomology, we can use suitable projectors to obtain compactly supported classes (however, one might need to work over a smaller affinoid neighbourhood of the weight space or to base change to $\Frac(\Lambda)$ possibly enlarging $\Lambda$ to contain suitable coefficients). An example relevant for our situation is the following: if $\mathbbm{f}$ is a Hida family corresponding to a $p$-stabilisation of ordinary cuspidal Hilbert eigenform of level $\NN$ coprime to $p$ the $\mathbbm{f}$-isotypical component $pr_{\mathbbm{f}}(\mathscr{Z}_{\infty})$ is a class in the parabolic cohomology $\Image(H^{2d}_{Iw,c}(Q_G)\to H^{2d}_{Iw}(Q_G))\otimes\Frac(\Lambda)$. Therefore, one obtains the following result combining Theorem \ref{dualitymain} and Lemma \ref{selfdual}:

\begin{prop}
    By applying $d_{Iw}$ to $\mathscr{Z}_{\infty}\otimes pr_{\mathbbm{f}} (\mathscr{Z}_{\infty})$ one obtains a function which at the weights $(\underline{k},0) \in (\mathbb{Z}_{\ge 0})^{\Sigma_L} \times \mathbb{Z},$
such that $k_\sigma \equiv 0 \pmod{2}, k_{\sigma}=k_{\bar{\sigma}}$ interpolates $p$-adically the intersection number of $\AL(\mathscr{Z}_{\infty})$ with the $\mathbbm{f}$-isotypical part of $\mathscr{Z}_{\infty}$ normalised by $p^{\sum k_{\sigma}}$.
\end{prop}

\bibliographystyle{amsalpha} 
\bibliography{references}

\end{document}